\documentclass{amsart}
\usepackage{amsfonts, amsmath, amssymb, mathrsfs, amscd, color}
\usepackage[active]{srcltx}
\usepackage{verbatim}

\theoremstyle{plain}
\newtheorem{theorem}{Theorem}[section]
\theoremstyle{remark}

\newtheorem{example}[theorem]{Example}
\theoremstyle{plain}

\newtheorem{lemma}[theorem]{Lemma}
\newtheorem{proposition}[theorem]{Proposition}

\numberwithin{equation}{section}

\newcommand{\R}{\mathbb{R}}

\newcommand{\nN}{n \in \mathbb{N}}
\newcommand{\N}{\mathbb{N}}
\newcommand{\C}{\mathbb{C}}
\newcommand{\E}{\mathbb{E}}

\newcommand{\Z}{\mathbb{Z}}

\newcommand{\Dom}{\mathsf{D}}

\newcommand{\bean}{\begin{eqnarray*}}
\newcommand{\eean}{\end{eqnarray*}}

\newcommand{\la}{\lambda}
\newcommand{\lb}{\langle}
\newcommand{\rb}{\rangle}
\newcommand{\wh}{\widehat}

\newcommand{\n}{\Vert}
\newcommand{\s}{^*}
\newcommand{\e}{\varepsilon}
\newcommand{\limn}{\lim_{n\to\infty}}

\renewcommand{\odot}{\hbox{\tiny\textcircled{s}}}

\newcommand{\calF}{\mathscr{F}}

\renewcommand{\P}{\mathbb{P}}

\newcommand{\one}{\mathbf{1}}

\renewcommand{\H}{{\mathscr H}}
\newcommand{\ga}{\gamma}

\allowdisplaybreaks

\begin{document}

\title[Second quantisation for infinite divisible measures] {Second Quantisation
for Skew Convolution Products of Infinitely Divisible Measures}

\author{David Applebaum}

\email{D.Applebaum@sheffield.ac.uk}
\address{School of Mathematics and Statistics\\
University of Sheffield \\
Sheffield S3 7RH\\
United Kingdom}

\author{Jan van Neerven}
\email{J.M.A.M.vanNeerven@tudelft.nl}
\address{Delft Institute of Applied Mathematics\\
Delft University of Technology\\
PO Box 5031\\ 2600 GA Delft \\
The Netherlands}

\date\today

\begin{abstract} Suppose $\la_1$ and $\la_2$ are infinitely divisible Radon measures
on real Banach spaces $E_1$ and $E_2$, respectively and let $T:E_{1} \rightarrow E_{2}$
be a Borel measurable mapping so that $T(\la_1) * \rho = \la_2 $ for some
Radon probability measure $\rho$ on $E_{2}$. Extending previous results for
the Gaussian and the Poissonian case, we study the problem of representing the
`transition operator' $P_{T}:L^{p}(E_{2}, \la_{2}) \rightarrow
L^{p}(E_{1}, \la_{1})$ given by
$$ P_{T}f(x) = \int_{E_{2}}f(T(x) + y)d\rho(y)
$$
as the second quantisation of a contraction operator acting between suitably chosen
`reproducing kernel Hilbert spaces' associated with $\la_1$ and $\la_2$.
\end{abstract}

\keywords{Second quantisation, infinitely divisible measure,
Wiener-It\^o decomposition, Poisson random measure}

\subjclass[2000]{Primary 60B11, Secondary 60E07, 60G51, 60G57, 60H05}

\thanks{The second named author was supported by VICI subsidy
639.033.604 of the Netherlands Organisation for Scientific Research (NWO)}

\maketitle

\section{Introduction}

Let $E_{i}$ $(i = 1,2)$ be real Banach spaces equipped with Radon probability measures
$\la_{1}$ and
$\la_{2}$, respectively. A Borel measurable mapping $T:E_{1} \rightarrow E_{2}$
is called a {\it skew map} for the pair $(\la_{1}, \la_{2})$ if there exists a
Radon probability measure $\rho$ on $E_{2}$ so that $\la_{2}$ is the convolution
of $\rho$ with the image of $\la_{1}$ under the action of $T$:
$$T(\la_1) * \rho = \la_2 .$$ In this case
for each $1\le p<\infty$ we
obtain a linear contraction  $P_{T}:L^{p}(E_{2}, \la_{2}) \rightarrow
L^{p}(E_{1}, \la_{1})$ given by
$$ P_{T}f(x) = \int_{E_{2}}f(T(x) + y)d\rho(y).$$
Such constructions arise naturally in the study of Mehler semigroups, linear
stochastic partial differential equations driven by additive L\'{e}vy noise, and
operator self-decomposable measures (see, e.g., \cite{BRS, DL, DLSS, Ju1, JV, SchSun}). In this context, the
problem of ``second quantisation'' is to find a functorial manner of expressing
$P_{T}$ in terms of $T$. The reason for this name is that the first work on this
subject \cite{CMG1}, within the context of Gaussian measures, exploited
constructions that were similar to those that are encountered in the
construction of the free quantum field from one-particle space (see e.g.
\cite{Par}) wherein the $n$th chaos spanned by multiple Wiener-It\^{o} integrals
corresponds to the $n$-particle space within the Fock space decomposition.

In our previous paper \cite{AppNee} we implemented this programme and
constructed $P_{T}$ as the second quantisation of $T$ in the two cases where
the measure $\la_{i}$ are Gaussian (generalising \cite{CMG1} and \cite{VN1}), and
are infinitely divisible measures of pure jump type (generalising \cite{Pesz}).
In this article, we complete the programme by dealing with the case where the
$\la_{i}$ are general infinitely divisible measures. Recall that a Radon probability  measure $\la$
on $E$ is said to be {\em infinitely divisible} of for each integer $n\ge 1$ there
exists a Radon probability  measure $\la_{1/n}$ whose $n$-fold convolution
product equals $\lambda$:
$$ \underbrace{\la_{1/n} * \hdots * \la_{1/n}}_{n \ {\rm times}} = \la.$$
These measures $\la_{1/n}$ are unique.

It is well-known that an infinitely divisible Radon probability
measure $\la$ on $E$ admits a unique representation as the convolution
\begin{align}\label{eq:LK} \la = \delta_\xi * \ga * \tilde e_{\rm s}(\nu),
\end{align}
where $\delta_\xi$ is the Dirac measure concentrated at the point $\xi\in E$, $\gamma$ is a centred
Gaussian Radon measure on $E$, and $\tilde e_{\rm s}(\nu)$ is the generalised exponential
of a Radon L\'evy measure $\nu$ on $E$ as in \cite[Theorem 3.4.20]{Hey}.

It is useful to rewrite (\ref{eq:LK}) from the point of view of random variables, rather than measures.
By \cite[Theorem 2.39]{Hey} there exists a semigroup of Radon probability measures $(\la_t)_{t\ge 0}$ such
that $\la = \la_1$. By the celebrated Kolmogorov construction (see, e.g., \cite[pp. 64--5]{App2})
we may construct an $E$-valued process $(X_t)_{t\ge 0}$ such that the
law of $X_t$ is $\la_t$ for each $t\ge 0$.
Using the L\'evy-It\^o decomposition of
Riedle and van Gaans \cite{RvG}, for $t=1$ we may then write
$$ X_1 = \xi + Q + \int_E x\,d\bar{\Pi}(x),$$
where $\xi\in E$ is as in \eqref{eq:LK}, $Q$ is the covariance of $\ga$,
and $\Pi$ is a Poisson random measure whose intensity measure $\nu$ is a L\'evy measure on $E$
and
$$\bar{\Pi}(dx) := \one_{\{0 < \n x\n\le 1\}}{\wh\Pi}(dx) + \one_{\{\n x\n
>1\}}{\Pi}(dx),$$
with $\wh \Pi$ the compensated Poisson random measure,
$$ \wh\Pi(B) := \Pi(B) - \nu(B).$$
In this description, the measure $\tilde e_{\rm s}(\nu)$ is the law of $\int_E x\,d\bar{\Pi}(x)$.

The data $\xi$, $\ga$, $\nu$ are uniquely determined by $\la$.
For more details we refer to \cite{AppNee, Lin, Pesz, RvG}.
In what follows we shall write
\begin{align}\label{eq:pi}\pi:= \delta_\xi * \tilde e_{\rm s}(\nu)
\end{align}
 for brevity.

From \cite{AppNee}, we know that we can effectively realise the second
quantisation of skew maps of $\ga$ in the symmetric Fock space $\Gamma(H_\ga)$
of the reproducing kernel Hilbert space $H_\ga$ of $\ga$; by the Wiener-It\^o chaos
decomposition this space is isomorphic
to $L^{2}(E,\ga)$. To second quantise skew maps of $\pi$, we use the fact that a similar
result holds if instead of the symmetric Fock space over $H_\ga$, we
consider the symmetric Fock space over $L^{2}(E,\nu)$; this is precisely
the approach adopted by Peszat in \cite{Pesz}.

The independence of $X_\ga$ and $X_\pi$ then suggests that in order to
unify these two approaches one should use the
symmetric Fock space over $H_\ga\oplus L^{2}(E,\nu)$. As we shall demonstrate
in this paper, this intuition is correct.

We finish this introduction by fixing some notation. All vector spaces
are real. Unless otherwise stated, Banach spaces are denoted by $E$, $F, \dots,$
and Hilbert spaces by $H$. The dual of a Banach space $E$ is denoted by $E^*$;
the duality pairing between vectors $x\in E$ and $x\s\in E^*$ is written
as $\lb x,x\s\rb$. Using the Riesz representation theorem, the dual of
a Hilbert space $H$ will always be identified with $H$ itself. The Fourier transform of a Radon probability measure
$\mu$ defined on $E$ is the mapping $\widehat{\mu}:E^{*} \rightarrow \C$ for which
$$ \widehat{\mu}(x\s) = \int_{E}e^{i\lb x,x\s\rb}d\mu(x).$$

\section{Skew Convolution Products of Infinitely Divisible Measures}\label{sec:prelim}

We fix two infinitely divisible Radon probability measures $\la_1$ and $\la_2$,
on the Banach spaces $E_1$ and $E_2$ respectively. We furthermore assume that a Borel linear mapping
$T: E_1\to E_2$ is given.
The main result of this section gives a necessary and sufficient condition in order that $T$
be skew with respect to the pair $(\la_1,\la_2)$.

We recall the L\'evy-Khintchine decompositions $\la_i = \ga_i * \pi_i$ of \eqref{eq:LK} and \eqref{eq:pi} (for $i = 1,2$)

\begin{proposition}\label{prop:skew}
 Under these assumptions the following assertions are equivalent:

 \begin{enumerate}
  \item[\rm(1)] $T$ is skew with respect to $(\la_1,\la_2)$ with an infinitely divisible skew factor;
  \item[\rm(2)] $T$ is skew with respect to both $(\ga_1,\ga_2)$ and $(\pi_1,\pi_2)$ with infinitely divisible skew factors.
 \end{enumerate}
 If these equivalent conditions are satisfied, the skew factor $\rho$ in {\rm (1)} and the skew factors
 $\rho_\ga$ and $\rho_\pi$ in {\rm (2)} are related by
 $ \rho = \rho_\ga*\rho_\pi.$
\end{proposition}
\begin{proof}
We begin by making the preliminary observation
that if $\alpha$ and $\beta$ are measures on $E_1$, then
their image measures under $T$ satisfy
$T(\alpha*\beta) = (T\alpha) *( T\beta).$
We shall freely use the properties of infinitely divisible measures
on Banach space as can be found in \cite{Hey,Lin}.

(2)$\Rightarrow$(1): \
From
$$ T\la_1* (\rho_\ga*\rho_\pi) = (T\ga_1* T\pi_1)* (\rho_\ga*\rho_\pi)
=  (T\ga_1* \rho_\ga)* (T\pi_1* \rho_\pi)
= \ga_2 * \pi_2 = \la_2$$
we infer that $T$ is skew for $(\la_1,\la_2)$ with skew factor $\rho_\ga*\rho_\pi$. This measure, being the
convolution of two infinitely divisible measures, is infinitely divisible.

\smallskip
(1)$\Rightarrow$(2): \
By the L\'evy-Khintchine decomposition theorem we have
$\la_i = \delta_{x_i} * \gamma_i * \tilde e_{\rm s}(\nu_i)$  $(i=1,2)$
using the notation introduced before
we have
$$ \la_1* \la_2 = (\delta_{\xi_1} * \gamma_1 * \tilde e_{\rm s}(\nu_1))*(\delta_{\xi_2} * \gamma_2 * \tilde e_{\rm s}(\nu_2))
= \delta_{\xi_1+\xi_2} * (\gamma_i *\ga_2) * \tilde e_{\rm s}(\nu_i+\nu_2).$$
By the uniqueness part of \cite[Theorem 3.4.20]{Hey}, this shows that the Gaussian factor of
$\la_1*\la_2$ equals $\ga_1*\ga_2$.

Now suppose that $T\la_1*\rho = \la_2$ with each of the measures $\la_1,\la_2$, and $\rho$ infinitely divisible.
Then $T\la_1$ is infinitely divisible with $T\la_1 = \delta_{Tx_1}*T\ga_1* T\tilde \e_{\rm s}(\nu_1)$, and
applying the remark of the previous paragraph to $T\la_1$ and $\rho$ we find that the Gaussian factor
of $T\la_1*\rho$ equals $T\ga_1*\eta$, where $\eta$ is the Gaussian factor of $\rho$. It follows
that
$$ T\ga_1*\eta = \ga_2,$$
 that is, $T$ is skew with respect to $(\ga_1,\ga_2)$
with Gaussian factor $\eta$. Taking Fourier transforms, this means that
\begin{equation} \label{use}
\widehat{T\ga_1}\widehat \eta = \widehat \ga_2.
\end{equation}

Finally, taking Fourier transforms in the original identity $T\la_1*\rho=\la_2$ we obtain
$\widehat{T\gamma_1}\widehat{T\pi_1}\widehat{\rho} = \widehat{\ga_2}\widehat{\pi_2}$
or equivalently, utilising (\ref{use})
$$ \widehat{T\pi_1}\Big(\frac{\widehat{T\gamma_1}}{\widehat{\ga_2}}\widehat{\rho}\Big) =
\widehat{T\pi_1}  \widehat{\eta} \widehat{\rho}= \widehat{T\pi_1}  \widehat{\eta *\rho}=\widehat{\pi_2}.$$
 From this we see that $T$ is skew with respect to $(\pi_1,\pi_2)$, with skew factor
 $\eta*\rho$.
\end{proof}

It is not true in general that $\mu_1*\mu_2 = \mu_3$ with $\mu_1$ and $\mu_3$ infinite divisible implies
the infinite divisibility of $\mu_2$. The following counterexample (in the case $E = \R$) is due to Jan Rosi\'nski who
kindly kindly permitted its inclusion here.

\begin{example}[Rosi\'nski]
Consider the signed measure $\nu:= 2\delta_1 + 2\delta_2 - \delta_3 + 2\delta_4 +2\delta_5$, where $\delta_x$ is the usual Dirac mass at $x \in \R$.
We claim that $$\phi(t):= \exp\Big(\int_0^\infty  (e^{itx}- 1) \,d\nu(x)\Big)$$
is the characteristic function of some non-negative random variable $Z$. This random variable cannot be infinitely
divisible. Indeed, if it were, $\nu$ would be its L\'evy measure, which is impossible because a L\'evy
measure is non-negative and unique. Therefore, to complete a counterexample we need to show
that $\phi$ is a characteristic function. Consider
$$
e(\nu) :=
\sum_{n=0}^\infty \frac{\nu^{*n}}{n!} .$$
First we compute
$$ \nu^{*2} =  4\delta_2 + 8\delta_3 + 4\delta_5 + 17\delta_6 + 4\delta_7 + 8\delta_9 + 4\delta_{10} $$
and
$$ \nu^{*3} = 8\delta_3 + 24\delta_4 + 12\delta_5 + 8\delta_6 + 66\delta_7 + 54\delta_8 - \delta_9 + 54\delta_{10}
+ 66\delta_{11} + 8\delta_{12} + 12\delta_{13} + 24\delta_{14} + 8\delta_{15} .$$
We have
$$ \nu^{*2}\ge 0, \quad  \nu+\frac18\nu^{*2}\ge 0, \quad \nu^{*2}+c\nu^{*3}\ge 0 \quad (0\le c\le 1).$$
Hence
$$
e(\nu) = \delta_0 + (\nu+\frac13\nu^{*2}) + \frac16(\nu^{*2}+\nu^{*3}) +
\sum_{n=2}^\infty \frac{\nu^{*2(n-1)}}{(2n)!} *\Big(\nu^{*2} + \frac{\nu^{*3}}{2n+1}\Big).$$
Consequently, $e(\nu)$ is a finite non-negative measure on $\Z_+$ with $(e(\nu))(\Z_+) = e^{\nu(\Z_+ )} = e^7.$
Take $Z$ to be a random variable with distribution $e^{-7}e(\nu)$. Then the characteristic function of $Z$
equals $\phi$. Now let $X$ be a compound Poisson random variable, independent of $Z$, and with L\'evy measure $\delta_3$.
Then $X+Z$ is compound Poisson with L\'evy measure $ 2\delta_1 + 2\delta_2+ 2\delta_4 + 2\delta_5$.
\end{example}

An interesting case where infinite divisibility of the skew factors is automatic
 occurs in the context of Mehler semigroups; we refer to  \cite{SchSun} for the details.

\section{Second Quantisation}

Suppose $\la$ is an infinitely divisible Radon measure on a real Banach space $E$.
Then we may write $$\la = \gamma*\pi$$ with $\gamma$ a centred Gaussian Radon measure on $E$ and $\pi$ the distribution
of a random variable of the form
$\xi + \int_E x\,d\bar{\Pi}(x)$ as explained in the introduction.

For functions $f\in L^2(\la)$ put
$$F_f(x,y):= f(x+y), \qquad x,y\in E.$$
Using the fact that
$ L^2(\gamma)\widehat\otimes L^2(\pi) = L^2(\gamma\times \pi)$ isometrically
(with $\widehat\otimes$ indicating the Hilbert space tensor product)
it is immediate to verify that
$$ \n f\n_{L^2(\la)}^2 =\int_E \int_E |f(x+y)|^2\,d\gamma(x)\,d\pi(y) = \n F_f\n_{L^2(\gamma)\widehat\otimes L^2(\pi)}^2.$$
As a result the mapping $f\mapsto F_f$ is a linear isometry from $L^2(\la)$ into $L^2(\gamma)\widehat\otimes L^2(\pi)$.
 This brings us to the setting with independence structure as discussed in
\cite{App3}.
Following that reference, formally we define a derivative operator acting with dense domain in
$L^2(\ga)\otimes L^2(\pi)$ by the formula
$$ D := D_\ga \otimes I + I \otimes D_\pi,$$
where we denote the `Gaussian' and the `Poissonian' derivatives with subscripts
$\ga$ and $\pi$, respectively. Recall from \cite{AppNee} that these are defined as follows. The Gaussian derivative is defined
by
$$ D_\gamma f(x) := \sum_{n=1}^N \partial_n g(\phi_{h_1}(x),\dots,\phi_{h_N}(x)) \otimes h_n$$
for cylindrical functions $f = g(\phi_{h_1},\dots,\phi_{h_N})$, with $g\in C_{\rm b}^1(\R^N)$ and
$\phi: h\mapsto \phi_h $ being the isometry
which embeds the reproducing kernel Hilbert space $H_\ga$ of $\ga$ onto the first Wiener-It\^o chaos of $L^2(\ga)$. The space of all such functions $f$ is
dense in $L^2(\ga)$ and $D_\gamma$ is closable as an operator from this initial domain into $L^2(\ga;H_\ga)$.
The Poissonian derivative is defined by
$$ D_\pi f (x)= f(x+\cdot) - f(x).$$
In order to prove that $D_\pi$ is densely defined as an operator from
$L^2(\pi)$ into $L^2(\pi\times \nu)$ we need to find a dense set of functions
$f$ in $L^2(\pi)$ such that
$D_\pi f$ belongs to $L^2(\pi\times \nu)$.
For this, we consider cylindrical functions $f$ of the form
$$ f(x) = g(\lb x,x_1\s\rb, \dots, \lb x,x_N\s\rb)$$
with $g\in C_{\rm b}^1(\R^N)$ and $x_1^*,\dots,x_n^*\in E^*$.
For such $f$ we have, where $0 < \theta_{n}(\cdot) < 1$ for each $\nN$,
\begin{align*}
\ & \n D_\pi f\n^2
\\ & = \int_{E\times E} \Big|
g(\lb x+y,x_1\s\rb, \dots, \lb x+y,x_N\s\rb) - g(\lb x,x_1\s\rb, \dots, \lb x,x_N\s\rb)
\Big|^2\,d\pi(x)\,d\nu(y)
\\ & =  \int\limits_{\{\n y\n> 1\}\times E} \Big|
g(\lb x+y,x_1\s\rb, \dots, \lb x+y,x_N\s\rb) - g(\lb x,x_1\s\rb, \dots, \lb x,x_N\s\rb)
\Big|^2 d\pi(x)\,d\nu(y)
\\ & \ + \!\!
\int\limits_{\{\n y\n \le 1\}\times E} \!\!\Big|\sum_{n=1}^N
(\partial_n g(\lb x+\theta_1(x)y,x_1\s\rb, \dots, \lb x+\theta_N(x)y,x_N\s\rb))\lb y,x_n\s\rb
\Big|^2 d\pi(x)\,d\nu(y)
\\ & \le 4 \n g\n_\infty^2 \nu\{\n y\n> 1\} +   \sum_{n=1}^N \n \partial_n g \n_\infty^2
\int_{\{\n y\n \le 1\}}  |\lb y,x_n\s\rb|^2\,d\nu(y)
\\ & < \infty,
\end{align*}
the finiteness in the last step
being a consequence of the general properties of L\'evy measures on Banach spaces (see \cite[pp. 95--120]{Hey}
or \cite[pp. 69--75]{Lin}).

\begin{lemma}
 $D_\pi$ is closable as a densely defined linear operator from $L^2(\pi)$ to $L^2(\pi\times \nu)$.
\end{lemma}
\begin{proof}
 Suppose $f_n \to 0$ in $L^2(\pi)$ and $D_\pi f_n \to F$ in $L^2(\pi\times \nu)$. We must prove that $F = 0$.
 Passing to a subsequence, we may assume that $f_n(x) \to 0$ for $\pi$-almost all $x\in E$
 and $D_\pi f_n(x,y)  = f_n(x+y)-f_n(x)\to F(x,y)$ for $\pi\times\nu$-almost all $(x,y)\in E\times E$.
 Then, by Fubini's theorem, for $\nu$-almost all $y\in E$ we have $f_n(x+y)\to F(x,y)$ for $\pi$-almost all $x\in E$.
 Since for all $y\in E$ we have $f_n(x+y)\to 0$ for $\pi$-almost all $x\in E$, it follows that for
 $\nu$-almost all $y\in E$ we have $F(x,y)=0$ for $\pi$-almost all $x\in E$. Using Fubini's theorem once more,
 it follows that $F(x,y)=0$ for $\pi\times\nu$-almost all $(x,y)\in E\times E$.
\end{proof}

From now on, we use the notations $D_\ga$ and $D_\pi$ for the closures of the operators considered so far
and denote by $\Dom(D_\ga)$ and $\Dom(D_\pi)$ their domains.

\begin{lemma}
 Suppose $T_1:E_1\to F_1$ and $T_2:E_2\to F_2$ are densely defined closed linear operators,
 with domains $\Dom(T_1)$ and $\Dom(T_2)$ respectively. Let $G$ be another Banach space and let $X\widehat\otimes Y$
 denote the completion of $X\otimes Y$ with respect to any
 norm which has the property that $\n x\otimes y\n = \n x\n \n y\n$ for all $x\in X$ and $y\in Y$.
\begin{itemize}
 \item [\rm(1)] The operators $T_1\otimes I: E_1\widehat\otimes G\to F_1\widehat\otimes G$
 and $I\otimes T_2: G\widehat\otimes E_2\to G\widehat\otimes F_2$ with their natural domains
 $\Dom(T_1)\otimes G$ and $G\otimes \Dom(T_2)$
 are closable;
 \item [\rm(2)] The operator $T_1\otimes I + I\otimes T_2: E_1\widehat\otimes E_2 \to F_1\widehat\otimes F_2$
 with its natural domain $\Dom(T_1)\otimes \Dom(T_2)$ is closable.
\end{itemize}
\end{lemma}
\begin{proof} Part (1) is immediate from the fact that $\n x\otimes y\n = \n x\n \n y\n$; part (2)
follows from the fact that a densely defined linear operator is closable if and  only if its
domain is weak$\s$-densely defined, along with the operator inclusion
$T_1^*\otimes I + I \otimes T_2\s \subseteq (T_1\otimes I + I \otimes T_2)\s$.
The details are left to the reader.
\end{proof}

To proceed any further we define, for $n=0,1,2,\dots$, the Hilbert spaces
$$\H_n := \bigoplus_{{\substack{j,k\ge 0 \\ j+k=n}}} H_\ga^{\odot j} \widehat\otimes L^2(\nu)^{\odot k}.$$
We use the convention that $G^{\odot0} = \R$ for any Hilbert space $G$ and
recall that $\widehat \otimes$ refers to the Hilbertian completion of the algebraic tensor
product.
We set
$$\H:= \H_1 = (H_\ga\widehat\otimes \R) \oplus (\R\widehat \otimes L^2(\nu)) = H_\ga \oplus L^2(\nu).$$

Having defined $D_\ga$ (respectively $D_\pi$) as closed densely defined operators
from $L^2(\ga)$  into $L^2(\ga)\widehat\otimes H_\ga$ (respectively from
$L^2(\pi)$ into $L^2(\pi\times\nu) =
L^2(\pi)\widehat \otimes L^2(\nu)$),
we now identify both $L^2(\ga)\widehat\otimes H_\ga$ and $L^2(\pi)\widehat \otimes L^2(\nu)$
canonically with closed subspaces of $(L^2(\ga)\widehat \otimes L^2(\pi))\widehat \otimes (H_\ga \oplus L^2(\nu))
= L^2(\ga\times \pi;\H)$. We denote by $D_\ga\otimes I$ and $I\otimes D_\pi$ the resulting closed and densely defined
operators from $L^2(\ga)\widehat\otimes L^2(\pi) = L^2(\ga\times \pi)$ into $L^2(\ga\times \pi;\H)$, and define
$$ D = D_\ga\otimes I + I\otimes D_\pi.$$
By part (1) of the previous lemma, after completing we can consider $D_\ga\otimes I$ and $I\otimes D_\pi$ as closed and densely defined
operators from  $L^2(\ga\times \pi;\H_n)$ into $L^2(\ga\times \pi;\H_{n+1})$,

By combining the preceding two lemmas we obtain the following result.
\begin{proposition} For all $n=0,1,2,\dots$, the operator $D = D_\ga\otimes I + I\otimes D_\pi$
is closable as a densely defined operator from $ L^2(\ga\times \pi;\H_n)$ into $L^2(\ga\times \pi;\H_{n+1})$.
\end{proposition}

We define the $n$-fold stochastic integral on $I_n: \H_n\to L^2(\Omega)$ by
$$ I_n (f\otimes g) :=  I_{j,\gamma} f \otimes I_{k,\pi}g$$
for $f\in H_\ga^{\odot j}$ and $g\in L^2(\nu)^{\odot k}$ with $j+k=n$,
where we denote the `Gaussian' and the `Poissonian' integrals with subscripts
$\ga$ and $\pi$, respectively.

In what follows, in order to tidy up the notation
we will refrain from writing subscripts $\gamma$ and $\pi$;
expectations taken in the the left and right sides of tensor products
refer to $\gamma$ and $\pi$, respectively.

Let $\Pi$ be a Poisson random measure on a probability space $(\Omega,\calF,\P)$,
whose intensity measure $\nu$ is a L\'evy measure on $E$.
Recall that the former means that $\Pi$ is a random
variable on $(\Omega,\calF,\P)$ taking values in the space $\N(E)$ of
$\N$-valued measures on $E$ endowed with the $\sigma$-algebra generated
by the Borel sets of $E$, that is, the smallest $\sigma$-algebra which renders the mappings
$\xi\mapsto \xi(B)$
measurable for all $B\in \mathscr{B}(E)$.
By $\P_\Pi$
we denote the image measure of $\P$ under $\Pi$.

Following Last and Penrose \cite{LP}, for a measurable function
$f: \N(Y)\to \R$ and $y\in Y$ we define the measurable function
$\tilde D_y f: \N(Y)\to \R$ by
$$\tilde D_y f (\eta):= f(\eta+\delta_y) - f(\eta).$$
The function $\tilde D_{y_1,\hdots,y_n}^nf:\N(Y)\to\R$ is defined recursively by
$$ \tilde D^n_{y_1,\hdots,y_n}f = \tilde D_{y_n} \tilde D^{n-1}_{y_1,\hdots,y_{n-1}}f,$$ for $y_{1}, \ldots, y_{n} \in Y$.
This function is symmetric, i.e. it is invariant under any permutation of the variables.

Following \cite{AppNee},
we define $j: L^2(E, \mu)\to L^2(\P_{\Pi})$ by
 $$ j f(\eta) = f\Big(\xi + \int_{E} x \,\bar{\eta}(dx)\Big),
\quad \eta\in \N(E).$$
The rigorous interpretation of this identity is provided by noting that
$$ \n j f\n_{L^2(\P_\Pi)}^2 = \E \Big| f\Big(\xi +
\int_{E} x \,d\bar{\Pi}(x)\Big)\Big|^2 = \n
f\n_{L^2(E,\mu)}^2,$$
which means that $j f(\eta)$ is well-defined for $\P_\Pi$-almost all $\eta$
and that $j$ establishes an isometry from $ L^2(E, \mu)$ into
$L^2(\P_{\Pi}).$
Note that
$$ j f (\Pi) = f\Big(\xi +
\int_{E} x \,d\bar{\Pi}(x)\Big)$$
and $$ j \circ D = \tilde D \circ j.$$

We now have the
following extension to infinitely divisible measures of the corresponding
results of Stroock \cite{Str} (for Gaussian measures) and
 Last and Penrose \cite{LP} (for Poisson random measures):

\begin{proposition} \label{prop:S-LP}
 For all $f\in W^{\infty,2}(\ga)$ and $g\in L^2(\P_\Pi)$ we have
$$ f\otimes g(\Pi)= \sum_{m=0}^\infty \frac1{m!} I_m (\E (\tilde D^m f\otimes g(\Pi))).$$
\end{proposition}
\begin{proof}
By Leibniz's rule,
\begin{align*}
 \sum_{m=0}^\infty \frac1{m!} I_m \E_\la \tilde D^m f\otimes g(\Pi)
  & = \sum_{m=0}^\infty \frac1{m!} I_m \Big(\E_\la \sum_{\ell=0}^m \binom{m}{\ell}
 \tilde D^\ell f \otimes \tilde D^{m-\ell}g(\Pi)\Big)
 \\ & =  \sum_{m=0}^\infty \sum_{\ell=0}^m \frac1{\ell!(m-\ell)!} I_m \Big(\E_\la
 \big(\tilde D^\ell f \otimes \tilde D^{m-\ell}g(\Pi) \big) \Big)
 \\ & = \sum_{m=0}^\infty \sum_{\ell=0}^m \frac1{\ell!(m-\ell)!} I_{\ell} (\E_\la
 \tilde D^\ell f) \otimes I_{m-\ell} (\E_\pi \tilde D^{m-\ell}g(\Pi))
 \\ &  = \sum_{j=0}^\infty \frac1{j!} I_{j} (\E_\ga  \tilde D^j f) \otimes
 \sum_{k=0}^\infty \frac1{k!} I_{k} (\E_\pi  \tilde D^k g(\Pi))
 \\ & = f\otimes g(\Pi).
\end{align*}
using the Stroock and Last-Penrose type decompositions in the penultimate identity.
\end{proof}

We now return to the setting considered in Section \ref{sec:prelim}
and make the standing assumption that the equivalent conditions stated in Proposition \ref{prop:skew} are satisfied.
Thus we assume that $\la_1 = \ga_1*\pi_1$ on $E_1$, $\la_2 = \ga_2*\pi_2$ on $E_2$, and
that $T:E_1\to E_2$ is a Borel linear
skew mapping with respect to the pair $(\la_1,\la_2)$ with an infinite divisible skew factor.
As is shown by Proposition \ref{prop:skew}, this implies that $T$ is skew with respect to
both pairs $(\ga_1,\ga_2)$ and $(\pi_1,\pi_2)$, that is,
$T\ga_1* \rho_\ga = \ga_2$ and $ T\pi_1* \rho_\pi = \pi_2$.

It follows Proposition \ref{prop:skew} that we may define $P_T: L^2(E_2,\la_2)\to L^2(E_1,\la_1)$  by
$$ P_T f(x) := \int_{E_2} f(Tx+y)\,d\rho(y), \quad x\in E_1,$$
where $\rho:=  \rho_\ga*\rho_\pi$ is the skew factor on $E_2$, i.e., $T\la_1*\rho = \la_2$.
Similarly we can define an operator $P_T\otimes P_T:
L^2(\gamma_2)\otimes L^2(\pi_2)\to L^2(\gamma_1)\otimes L^2(\pi_1)$ in the obvious way
(with slight abuse of notation; we should really be writing
$P_{\ga,T}\otimes P_{\pi,T}$) and we then have:

\begin{lemma} Under the above assumptions,
$F_{P_T f} = (P_T\otimes P_T)F_f$.
\end{lemma}
\begin{proof}
For $(\gamma\times\pi)$-almost all $x,y\in E_2$ we have
\begin{align*}
 (P_T\otimes P_T)(\phi\otimes \psi)(x,y) & = (P_T\phi\otimes P_T\psi)(x,y)
 \\ & = \int_{E_2} \phi(Tx+z)\,d\rho_\gamma(z)\int_{E_2} \psi(Ty+z)\,d\rho_\pi(z)
\\ &  = \int_{E_2}\int_{E_2} (\phi\otimes \psi)(Tx+z_1, Ty+z_2)\,d\rho_\ga(z_1)\,d\rho_\pi(z_2).
 \end{align*}
Now suppose that $F_f = \limn G_n$ in $L^2(\gamma\times \pi)$, where each $G_n$ belongs
to the algebraic tensor product $L^2(\gamma)\otimes L^2(\pi)$. By the above identity and linearity it follows,
after passing to a subsequence if necessary, that for $(\gamma\times\pi)$-almost all $x,y\in E_2$ we have
\begin{align*}(P_T\otimes P_T) F_f(x,y)
& = \limn  (P_T\otimes P_T)G_n(x,y)
\\ & = \limn \int_{E_2}\int_{E_2} G_n (Tx+z_1, Ty+z_2)\,d\rho_\ga(z_1)\,d\rho_\pi(z_2)
\\ & =  \int_{E_2}\int_{E_2} F_f(Tx+z_1, Ty+z_2)\,d\rho_\ga(z_1)\,d\rho_\pi(z_2)
\\ & =  \int_{E_2}\int_{E_2} f(Tx+Ty+z_1+z_2)\,d\rho_\ga(z_1)\,d\rho_\pi(z_2)
\\ & =  \int_{E_2} f(Tx+ Ty+z)\,d(\rho_\ga*\rho_\pi)(z)
\\ & =  \int_{E_2} f(Tx+ Ty+z)\,d\rho(z)
\\ & =  P_T f(x+y)
\\ & = F_{P_T}(x,y).
\end{align*}
\end{proof}

For $h\in H_\ga$ and  $y_1,\dots, y_n\in E$ and $h\in H_\ga$ we define
$$ D_{h;y} := D_{h} \otimes I + I \otimes D_{y},$$
where
$$ D_{h} f (x):= \lb D_\gamma f(x),h\rb, \quad  D_{y}g(x):= (D_\pi g(x))(y).$$
For the higher order derivatives we define inductively
$$ D_{h_1,\dots,h_n;y_1,\hdots,y_n}^n := D_{h_n;y_n} D_{h_1,\dots,h_{n-1};y_1,\hdots,y_{n-1}}^{n-1}.$$
\begin{lemma}\label{lem:commute}
For all
$f\in L^2(E_{2},\la_2)$, $h\in H_\ga$, and $y_1,\dots, y_n\in
E_{1}$,
\begin{align}\label{eq:higher-der-Poisson}
\E_{\ga_1\times\pi_1} D_{h_1,\dots,h_n;y_1,\hdots,y_n}^n F_{P_T f} =\E_{\ga_2\times\pi_2}
D_{Th_1\dots,Th_n; Ty_1,\hdots,Ty_n}^n F_f.
\end{align}
\end{lemma}
\begin{proof}
We approximate $F_f$ by finite sums of elementary tensors as in the proof of the previous lemma.
For such functions $G_n$ the identity follows from the results in \cite{AppNee} for the Gaussian and Poissonian case.
Thanks to the closedness of the derivative operators, the identity passes over to the limit.
\end{proof}

For Hilbert spaces $H$ and $\underline H$ we note that
$$ \Gamma(H\oplus \underline H) = \bigoplus_{n=0}^\infty\Big(\bigoplus_{{\substack{j,k\ge 0 \\ j+k=n}}}
H^{\odot j}\widehat\otimes \underline H^{\odot k}\Big).$$

Putting everything together we obtain the following result which generalises the results of Theorems 3.5 and 4.4 of \cite{AppNee}, where Gaussian and Poisson noises were treated separately.

\begin{theorem} Under the standing assumption stated above, the following diagram commutes:
\begin{equation*}
  \begin{CD}
L^2(E_2,\la_2)     @> P_T >>  L^2(E_1,\la_1)  \\
   @V   f\mapsto F_f  VV
   @VV f \mapsto F_f   V \\
   L^2(E_2\times E_2,\ga_2\times \pi_2)  @> P_T\otimes P_T >> L^2(E_1\times E_1,\ga_1\times \pi_1)   \\
    @V  \bigoplus_{n=0}^\infty \frac1{\sqrt{n!}}\E_{\ga_2\times\pi_2} D^n   VV
    @VV \bigoplus_{n=0}^\infty \frac1{\sqrt{n!}}\E_{\ga_1\times\pi_1} D^n   V \\
\Gamma((H_{\ga,2}\oplus L^2(E_2,\nu_2))@ >\bigoplus_{n=0}^\infty (T^*)^{\odot n} >>\Gamma(H_{\ga,1}\oplus L^2(E_1,\nu_1))
      \end{CD}
 \end{equation*}
 Moreover, for $k=1,2$ also the following diagram commutes in distribution if $X_k$ is an $E$-valued random variable with
  distribution $\la_k$:
 \begin{equation*}
   \begin{CD}
 L^2(E_k\times E_k,\ga_k\times \pi_k)     @> (f,g)\mapsto (f(X_{\ga,k}), f(X_{\pi,k})) >>   L^2(\Omega\times\Omega) \\
     @V \bigoplus_{n=0}^\infty \frac1{\sqrt{n!}}\E_{\ga_k\times\pi_k} D^n   VV
     @AA \bigoplus_{n=0}^\infty \frac1{\sqrt{n!}}I_n A \\
  \Gamma(H\oplus L^2(E_k,\nu_k))@ > = >> \Gamma(H\oplus L^2(E_k,\nu_k))
       \end{CD}
  \end{equation*}
\end{theorem}

\medskip \

\end{document}